\documentclass{amsart}
\usepackage{amsmath}
\usepackage{amsfonts}

\newtheorem{theorem}{Theorem}
\theoremstyle{plain}

\newtheorem{definition}{Definition}

\newtheorem{lemma}{Lemma}

\newtheorem{remark}{Remark}

\numberwithin{equation}{section}
\input{tcilatex}

\begin{document}
\title[Convergence Analysis for Various Iterative Schemes]{Comparison of The
Speed of Convergence Among Various Iterative Schemes}
\author{Vatan KARAKAYA}
\address{Department of Mathematical Engineering, Yildiz Technical
University, Davutpasa Campus, Esenler, 34210 \.{I}stanbul, Turkey}
\email{vkkaya@yildiz.edu.tr; vkkaya@yahoo.com}
\urladdr{http://www.yarbis.yildiz.edu.tr/vkkaya}
\author{Faik G\"{U}RSOY}
\address{Department of Mathematics, Yildiz Technical University, Davutpasa
Campus, Esenler, 34220 Istanbul, Turkey}
\email{faikgursoy02@hotmail.com; fgursoy@yildiz.edu.tr}
\urladdr{http://www.yarbis.yildiz.edu.tr/fgursoy}
\author{M\"{u}zeyyen ERT\"{U}RK}
\address{Department of Mathematics, Yildiz Technical University, Davutpasa
Campus, Esenler, 34220 Istanbul, Turkey}
\email{merturk3263@gmail.com}
\urladdr{http://www.yarbis.yildiz.edu.tr/merturk}
\subjclass[2000]{Primary 47H06, 54H25.}
\keywords{Iterative schemes, Convergence, Speed of convergence, Equivalence
of convergence, Data dependence of fixed points.}

\begin{abstract}
We show that iterative scheme due to Karahan and \"{O}zdemir (2013) can be
used to approximate fixed point of contraction mappings. Furthermore, we
prove that CR iterative scheme converges faster than the iterative scheme
due to Karahan and \"{O}zdemir (2013) for the class contraction mappings.
Finally, we prove a data dependence result for contraction mappings by
employing iterative scheme due to Karahan and \"{O}zdemir (2013).
\end{abstract}

\maketitle

\section{introduction}

Fixed point theory has been appeared as one of the most powerful and
substantial theoretical tools of mathematics. This theory has a long history
and has been studied intensively by many researchers in various aspects. For
the past 30 years or so, the study of iterative procedures for the
approximation of fixed points of various classes of operators have been
flourishing areas of research for many mathematicians. Consequently,
considerable research efforts have been devoted to introduce various
iteration methods and study its more qualitative features, for example, \cite%
{S, KN, Ishikawa, Kirk, Krasnosel'skii, Mann, Noor, KM-I, Olatinwo, SP,
Picard, Multistep, GenKrasnosel'skii, iam}.

We begin our exposition with an overview of various iterative methods.

Throughout this paper $%
\mathbb{N}
$ denotes set of all nonnegative integers including zero. Let $B$ be a
Banach space, $S$ be a subset of $B$ and $T$ be a selfmap of $S$. Let $%
\left\{ \alpha _{n}^{i}\right\} _{n=0}^{\infty }$, $i\in \left\{
1,2,3\right\} $ be real sequences in $\left[ 0,1\right] $ satisfying certain
control condition(s).

An iterative sequence $\left\{ x_{n}\right\} _{n=0}^{\infty }$ defined by%
\begin{equation}
\left\{ 
\begin{array}{c}
x_{0}\in S\text{, \ \ \ \ \ \ \ \ \ \ \ \ \ \ \ } \\ 
x_{n+1}=Tx_{n}\text{, }n\in 
\mathbb{N}
\text{,}%
\end{array}%
\right.  \label{eqn1}
\end{equation}%
is known as Picard iteration procedure \cite{Picard}, which is commonly used
to approximate fixed point of contraction mappings satisfying%
\begin{equation}
\left\Vert Tx-Ty\right\Vert \leq \delta \left\Vert x-y\right\Vert \text{, }%
\delta \in \left( 0,1\right) \text{, for all }x,y\in B\text{.}  \label{eqn2}
\end{equation}%
The following iteration methods are called Noor \cite{Noor}, and SP \cite{SP}
iteration methods, respectively:%
\begin{equation}
\left\{ 
\begin{array}{c}
x_{0}\in S\text{, \ \ \ \ \ \ \ \ \ \ \ \ \ \ \ \ \ \ \ \ \ \ \ \ \ \ \ \ \
\ \ \ \ \ \ \ \ \ \ \ } \\ 
x_{n+1}=\left( 1-\alpha _{n}^{1}\right) x_{n}+\alpha _{n}^{1}Ty_{n}\text{, \
\ \ \ \ \ \ \ \ \ } \\ 
y_{n}=\left( 1-\alpha _{n}^{2}\right) x_{n}+\alpha _{n}^{2}Tz_{n}\text{, \ \
\ \ \ \ } \\ 
z_{n}=\left( 1-\alpha _{n}^{3}\right) x_{n}+\alpha _{n}^{3}Tx_{n}\text{, }%
n\in 
\mathbb{N}
\text{,}%
\end{array}%
\right.  \label{eqn3}
\end{equation}%
and%
\begin{equation}
\left\{ 
\begin{array}{c}
x_{0}\in S\text{, \ \ \ \ \ \ \ \ \ \ \ \ \ \ \ \ \ \ \ \ \ \ \ \ \ \ \ \ \
\ \ \ \ \ \ \ \ \ \ \ } \\ 
x_{n+1}=\left( 1-\alpha _{n}^{1}\right) y_{n}+\alpha _{n}^{1}Ty_{n}\text{, \
\ \ \ \ \ \ \ \ \ } \\ 
y_{n}=\left( 1-\alpha _{n}^{2}\right) z_{n}+\alpha _{n}^{2}Tz_{n}\text{, \ \
\ \ \ \ } \\ 
z_{n}=\left( 1-\alpha _{n}^{3}\right) x_{n}+\alpha _{n}^{3}Tx_{n}\text{, }%
n\in 
\mathbb{N}
\text{.}%
\end{array}%
\right.  \label{eqn4}
\end{equation}

\begin{remark}
(i) Noor iteration method (1.5) reduces to well-known Ishikawa iteration
method \cite{Ishikawa} for $\alpha _{n}^{3}=0$ and Mann iteration method 
\cite{Mann} for $\alpha _{n}^{2}=\alpha _{n}^{3}=0$. (ii) SP iteration
method (1.6) reduces to a two-step Mann iteration method \cite{iam} for $%
\alpha _{n}^{3}=0$.
\end{remark}

Agarwal et al. \cite{S} inroduced an S-iteration method as follows%
\begin{equation}
\left\{ 
\begin{array}{c}
s_{0}\in S\text{, \ \ \ \ \ \ \ \ \ \ \ \ \ \ \ \ \ \ \ \ \ \ \ \ \ \ \ \ \
\ \ \ \ \ \ \ \ \ \ \ } \\ 
s_{n+1}=\left( 1-\alpha _{n}^{1}\right) Ts_{n}+\alpha _{n}^{1}Tt_{n}\text{,
\ \ \ \ \ \ \ \ \ \ } \\ 
t_{n}=\left( 1-\alpha _{n}^{2}\right) s_{n}+\alpha _{n}^{2}Ts_{n}\text{, }%
n\in 
\mathbb{N}
\text{.}%
\end{array}%
\right.  \label{eqn5}
\end{equation}%
The following iteration method is referred to as CR iteration method \cite%
{CR} 
\begin{equation}
\left\{ 
\begin{array}{c}
u_{0}\in S\text{, \ \ \ \ \ \ \ \ \ \ \ \ \ \ \ \ \ \ \ \ \ \ \ \ \ \ \ \ \
\ \ \ \ \ \ \ \ \ \ \ } \\ 
u_{n+1}=\left( 1-\alpha _{n}^{1}\right) v_{n}+\alpha _{n}^{1}Tv_{n}\text{, \
\ \ \ \ \ \ \ \ \ } \\ 
v_{n}=\left( 1-\alpha _{n}^{2}\right) Tu_{n}+\alpha _{n}^{2}Ty_{n}\text{, \
\ \ \ \ \ } \\ 
y_{n}=\left( 1-\alpha _{n}^{3}\right) u_{n}+\alpha _{n}^{3}Tu_{n}\text{, }%
n\in 
\mathbb{N}
\text{.}%
\end{array}%
\right.  \label{eqn6}
\end{equation}%
Very recently, Karahan and \"{O}zdemir \cite{Karahan} introduced a new three
step iteration as follows%
\begin{equation}
\left\{ 
\begin{array}{c}
p_{0}\in S\text{, \ \ \ \ \ \ \ \ \ \ \ \ \ \ \ \ \ \ \ \ \ \ \ \ \ \ \ \ \
\ \ \ \ \ \ \ \ \ \ \ } \\ 
p_{n+1}=\left( 1-\alpha _{n}^{1}\right) Tp_{n}+\alpha _{n}^{1}Tq_{n}\text{,
\ \ \ \ \ \ \ \ \ \ } \\ 
q_{n}=\left( 1-\alpha _{n}^{2}\right) Tp_{n}+\alpha _{n}^{2}Tr_{n}\text{, \
\ \ \ \ \ } \\ 
r_{n}=\left( 1-\alpha _{n}^{3}\right) p_{n}+\alpha _{n}^{3}Tp_{n}\text{, }%
n\in 
\mathbb{N}
\text{.}%
\end{array}%
\right.  \label{eqn7}
\end{equation}%
Convergence analysis of iterative methods has an important role in the study
of iterative approximation of fixed point theory. Fixed point iteration
methods may exhibit radically different behaviors for various classes of
mappings. While a particular fixed point iteration method is convergent for
an appropriate class of mappings, it may not be convergent for the others.
Due to various reasons, it is important to determine whether an iteration
method converges to fixed point of a mapping. In many cases, there can be
two or more than two iteration procedures approximating to a fixed point of
a mapping, for example, \cite{SS, Karakaya, Rhoades, Zxue}. In such cases,
the critical and important point is to compare rate of convergence of these
iterations to find out which ones converge faster to that fixed point, e.g., 
\cite{Babu, Berinde, Berinde1, Berinde2, Hussain1, Hussain, Suantai,
Popescu, Rhoades1, Sahu, ZXUE, Yuan}.

Recently, several authors introduced different type iteration methods and
they have proved that their iteration methods converges faster than Picard,
Mann and Ishikawa iteration methods, e.g., \cite{CR, Karahan, Karakaya,
Suantai, Sahu}.

In this paper, we are concerned with two recent iteration methods defined by
(1.6) and (1.7). We show that iteration method (1.7) converges to fixed
point of a contraction mapping satisfying (1.2). Also, we prove that CR
iteration method (1.6) is equivalent and faster than iteration method (1.7)
for the class of contraction mappings. Finally, we give a data dependence
result for the fixed point of contraction mappings using iteration method
(1.7).

In order to obtain our main results we need following lemmas and definitions.

\begin{definition}
\cite{Vasile} Let $\left\{ a_{n}\right\} _{n=0}^{\infty }$ and $\left\{
b_{n}\right\} _{n=0}^{\infty }$ be two sequences of real numbers with limits 
$a$ and $b$, respectively. Assume that there exists%
\begin{equation}
\underset{n\rightarrow \infty }{\lim }\frac{\left\vert a_{n}-a\right\vert }{%
\left\vert b_{n}-b\right\vert }=l\text{.}  \label{eqn8}
\end{equation}

(i) If $l=0$, the we say that $\left\{ a_{n}\right\} _{n=0}^{\infty }$
converges faster to $a$ than $\left\{ b_{n}\right\} _{n=0}^{\infty }$ to $b$.

(ii) If $0<l<\infty $, then we say that $\left\{ a_{n}\right\}
_{n=0}^{\infty }$ and $\left\{ b_{n}\right\} _{n=0}^{\infty }$ have the same
rate of convergence.
\end{definition}

\begin{definition}
\cite{Vasile} Let $T$,$\widetilde{T}:B\rightarrow B$ be two operators. We
say that $\widetilde{T}$ is an approximate operator of $T$ if for all $x\in
B $ and for a fixed $\varepsilon >0$ we have 
\begin{equation}
\left\Vert Tx-\widetilde{T}x\right\Vert \leq \varepsilon .  \label{eqn9}
\end{equation}
\end{definition}

\begin{lemma}
\cite{Weng}Let $\left\{ a_{n}\right\} _{n=0}^{\infty }$ and $\left\{ \rho
_{n}\right\} _{n=0}^{\infty }$ be nonnegative real sequences satisfying the
following inequality:%
\begin{equation}
a_{n+1}\leq \left( 1-\eta _{n}\right) a_{n}+\rho _{n}\text{,}  \label{eqn10}
\end{equation}%
where $\eta _{n}\in \left( 0,1\right) $, for all $n\geq n_{0}$, $%
\dsum\nolimits_{n=1}^{\infty }\eta _{n}=\infty $, and $\frac{\rho _{n}}{\eta
_{n}}\rightarrow 0$ as $n\rightarrow \infty $. Then $\lim_{n\rightarrow
\infty }a_{n}=0$.
\end{lemma}

\begin{lemma}
\cite{Data Is 2} Let $\left\{ a_{n}\right\} _{n=0}^{\infty }$ be a
nonnegative sequence for which one assumes there exists $n_{0}\in 
\mathbb{N}
$, such that for all $n\geq n_{0}$ one has satisfied the inequality 
\begin{equation}
a_{n+1}\leq \left( 1-\mu _{n}\right) a_{n}+\mu _{n}\eta _{n}\text{,}
\label{eqn11}
\end{equation}
where $\mu _{n}\in \left( 0,1\right) ,$ for all $n\in 
\mathbb{N}
$, $\sum\limits_{n=0}^{\infty }\mu _{n}=\infty $ and $\eta _{n}\geq 0$, $%
\forall n\in 
\mathbb{N}
$. Then the following inequality holds 
\begin{equation}
0\leq \lim \sup_{n\rightarrow \infty }a_{n}\leq \lim \sup_{n\rightarrow
\infty }\eta _{n}.  \label{eqn12}
\end{equation}
\end{lemma}

\section{Main Results}

\begin{theorem}
Let $S$ be a nonempty closed convex subset of a Banach space $B$ and $%
T:S\rightarrow S$ be a contraction map satisfying condition (1.2). Let $%
\left\{ p_{n}\right\} _{n=0}^{\infty }$ be an iterative sequence generated
by (1.7) with real sequences $\left\{ \alpha _{n}^{i}\right\} _{n=0}^{\infty
}$, $i\in \left\{ 1,2,3\right\} $ in $\left[ 0,1\right] $ satisfying $%
\sum_{k=0}^{n}\alpha _{k}^{1}=\infty $. Then $\ \left\{ p_{n}\right\}
_{n=0}^{\infty }$ converges to a unique fixed point of $T$, say $x_{\ast }$.
\end{theorem}

\begin{proof}
Picard-Banach theorem guarantees the existence and uniqueness of $x_{\ast }$%
. We will show that $p_{n}\rightarrow x_{\ast }$ as $n\rightarrow \infty $.
From (1.2) and (1.7) we have%
\begin{eqnarray}
\left\Vert r_{n}-x_{\ast }\right\Vert &=&\left\Vert \left( 1-\alpha
_{n}^{3}\right) p_{n}+\alpha _{n}^{3}Tp_{n}-\left( 1-\alpha _{n}^{3}+\alpha
_{n}^{3}\right) x_{\ast }\right\Vert  \notag \\
&\leq &\left( 1-\alpha _{n}^{3}\right) \left\Vert p_{n}-x_{\ast }\right\Vert
+\alpha _{n}^{3}\left\Vert Tp_{n}-Tx_{\ast }\right\Vert  \notag \\
&\leq &\left( 1-\alpha _{n}^{3}\right) \left\Vert p_{n}-x_{\ast }\right\Vert
+\alpha _{n}^{3}\delta \left\Vert p_{n}-x_{\ast }\right\Vert  \notag \\
&=&\left[ 1-\alpha _{n}^{3}\left( 1-\delta \right) \right] \left\Vert
p_{n}-x_{\ast }\right\Vert \text{,}  \label{eqn13}
\end{eqnarray}%
\begin{eqnarray}
\left\Vert q_{n}-x_{\ast }\right\Vert &\leq &\left( 1-\alpha _{n}^{2}\right)
\left\Vert Tp_{n}-Tx_{\ast }\right\Vert +\alpha _{n}^{2}\left\Vert
Tr_{n}-Tx_{\ast }\right\Vert  \notag \\
&\leq &\left( 1-\alpha _{n}^{2}\right) \delta \left\Vert p_{n}-x_{\ast
}\right\Vert +\alpha _{n}^{2}\delta \left\Vert r_{n}-x_{\ast }\right\Vert 
\notag \\
&\leq &\left\{ \left( 1-\alpha _{n}^{2}\right) \delta +\alpha _{n}^{2}\delta 
\left[ 1-\alpha _{n}^{3}\left( 1-\delta \right) \right] \right\} \left\Vert
p_{n}-x_{\ast }\right\Vert \text{,}  \label{eqn14}
\end{eqnarray}%
and%
\begin{eqnarray}
\left\Vert p_{n+1}-x_{\ast }\right\Vert &\leq &\left( 1-\alpha
_{n}^{1}\right) \left\Vert Tp_{n}-Tx_{\ast }\right\Vert +\alpha
_{n}^{1}\left\Vert Tq_{n}-Tx_{\ast }\right\Vert  \notag \\
&\leq &\left( 1-\alpha _{n}^{1}\right) \delta \left\Vert p_{n}-x_{\ast
}\right\Vert +\alpha _{n}^{1}\delta \left\Vert q_{n}-x_{\ast }\right\Vert 
\notag \\
&\leq &\left\{ \left( 1-\alpha _{n}^{1}\right) \delta \right.  \notag \\
&&\left. +\alpha _{n}^{1}\delta \left\{ \left( 1-\alpha _{n}^{2}\right)
\delta +\alpha _{n}^{2}\delta \left[ 1-\alpha _{n}^{3}\left( 1-\delta
\right) \right] \right\} \right\} \left\Vert p_{n}-x_{\ast }\right\Vert
\label{eqn15}
\end{eqnarray}%
Since $\delta \in \left( 0,1\right) $ and $\alpha _{n}^{i}\in \left[ 0,1%
\right] $, for all $n\in 
\mathbb{N}
$ and for each $i\in \left\{ 1,2,3\right\} $%
\begin{equation}
1-\alpha _{n}^{2}\left( 1-\delta \right) <1\text{,}  \label{eqn16}
\end{equation}%
\begin{equation}
1-\alpha _{n}^{3}\left( 1-\delta \right) <1\text{.}  \label{eqn17}
\end{equation}%
By using $\delta \in \left( 0,1\right) $, (2.4) and (2.5) in (2.3), we obtain%
\begin{eqnarray}
\left\Vert p_{n+1}-x_{\ast }\right\Vert &\leq &\left[ \left( 1-\alpha
_{n}^{1}\right) \delta \left\{ \left( 1-\alpha _{n}^{2}\right) \delta
+\alpha _{n}^{2}\delta \right\} +\alpha _{n}^{1}\delta \right] \left\Vert
p_{n}-x_{\ast }\right\Vert  \notag \\
&\leq &\left[ \left( 1-\alpha _{n}^{1}\right) \delta \left\{ 1-\alpha
_{n}^{2}\left( 1-\delta \right) \right\} +\alpha _{n}^{1}\delta \right]
\left\Vert p_{n}-x_{\ast }\right\Vert  \notag \\
&\leq &\left[ 1-\alpha _{n}^{1}\left( 1-\delta \right) \right] \left\Vert
p_{n}-x_{\ast }\right\Vert  \notag \\
&\leq &\cdots  \notag \\
&\leq &\prod\limits_{k=0}^{n}\left[ 1-\alpha _{k}^{1}\left( 1-\delta \right) %
\right] \left\Vert p_{0}-x_{\ast }\right\Vert \text{.}  \label{eqn18}
\end{eqnarray}%
It is well-known from the classical analysis that $1-x\leq e^{-x}$ for all $%
x\in \left[ 0,1\right] $. By considering this fact together with (2.6), we
obtain%
\begin{eqnarray}
\left\Vert p_{n+1}-x_{\ast }\right\Vert &\leq &\prod\limits_{k=0}^{n}\left[
1-\alpha _{k}^{1}\left( 1-\delta \right) \right] \left\Vert p_{0}-x_{\ast
}\right\Vert  \notag \\
&\leq &\frac{\left\Vert p_{0}-x_{\ast }\right\Vert }{e^{\left( 1-\delta
\right) \sum_{k=0}^{n}\alpha _{k}^{1}}}\text{.}  \label{eqn19}
\end{eqnarray}%
Taking the limit of both sides of inequality (2.7) yields $p_{n}\rightarrow
x_{\ast }$ as $n\rightarrow \infty $.
\end{proof}

\begin{theorem}
Let $S$, $B$ and $T$ with fixed point $x_{\ast }$ be as in Theorem 1. Let $%
\{u_{n}\}_{n=0}^{\infty }$, $\{p_{n}\}_{n=0}^{\infty }$be two iterative
sequences defined by (1.6) and (1.7) with real sequences $\left\{ \alpha
_{n}^{i}\right\} _{n=0}^{\infty }$, $i\in \left\{ 1,2,3\right\} $ in $\left[
0,1\right] $ satisfying $\sum_{k=0}^{n}\alpha _{k}^{1}=\infty $. Then the
following are equivalent:

(i) the iteration method (1.7) converges to the fixed point $x_{\ast }$ of $%
T $;

(ii) the CR iteration method (1.6) converges to the fixed point $x_{\ast }$
of $T$.
\end{theorem}

\begin{proof}
We will prove (i)$\Rightarrow $(ii), that is, if iteration method (1.7)
converges to $x_{\ast }$, then CR iteration method (1.6) does too. Now by
using iteration method (1.7), CR iteration method (1.6) and condition (1.2),
we have%
\begin{eqnarray}
\left\Vert p_{n+1}-u_{n+1}\right\Vert &=&\left\Vert \left( 1-\alpha
_{n}^{1}\right) Tp_{n}+\alpha _{n}^{1}Tq_{n}-\left( 1-\alpha _{n}^{1}\right)
v_{n}-\alpha _{n}^{1}Tv_{n}\right\Vert  \notag \\
&\leq &\left( 1-\alpha _{n}^{1}\right) \left\Vert Tp_{n}-v_{n}\right\Vert
+\alpha _{n}^{1}\left\Vert Tq_{n}-Tv_{n}\right\Vert  \notag \\
&\leq &\left( 1-\alpha _{n}^{1}\right) \left\Vert p_{n}-v_{n}\right\Vert
+\alpha _{n}^{1}\delta \left\Vert q_{n}-v_{n}\right\Vert +\left( 1-\alpha
_{n}^{1}\right) \left\Vert p_{n}-Tp_{n}\right\Vert  \notag \\
&=&\left( 1-\alpha _{n}^{1}\right) \left\Vert \left( 1-\alpha
_{n}^{2}+\alpha _{n}^{2}\right) p_{n}-\left( 1-\alpha _{n}^{2}\right)
Tu_{n}-\alpha _{n}^{2}Ty_{n}\right\Vert  \notag \\
&&+\alpha _{n}^{1}\delta \left\Vert q_{n}-v_{n}\right\Vert +\left( 1-\alpha
_{n}^{1}\right) \left\Vert p_{n}-Tp_{n}\right\Vert  \notag \\
&\leq &\left( 1-\alpha _{n}^{1}\right) \left( 1-\alpha _{n}^{2}\right)
\left\Vert p_{n}-Tu_{n}\right\Vert +\left( 1-\alpha _{n}^{1}\right) \alpha
_{n}^{2}\left\Vert p_{n}-Ty_{n}\right\Vert  \notag \\
&&+\alpha _{n}^{1}\delta \left\Vert q_{n}-v_{n}\right\Vert +\left( 1-\alpha
_{n}^{1}\right) \left\Vert p_{n}-Tp_{n}\right\Vert  \notag \\
&\leq &\left( 1-\alpha _{n}^{1}\right) \left( 1-\alpha _{n}^{2}\right)
\left\{ \left\Vert p_{n}-Tp_{n}\right\Vert +\left\Vert
Tp_{n}-Tu_{n}\right\Vert \right\}  \notag \\
&&+\left( 1-\alpha _{n}^{1}\right) \alpha _{n}^{2}\left\{ \left\Vert
p_{n}-Tp_{n}\right\Vert +\left\Vert Tp_{n}-Ty_{n}\right\Vert \right\}  \notag
\\
&&+\alpha _{n}^{1}\delta \left\Vert q_{n}-v_{n}\right\Vert +\left( 1-\alpha
_{n}^{1}\right) \left\Vert p_{n}-Tp_{n}\right\Vert  \notag \\
&\leq &\left( 1-\alpha _{n}^{1}\right) \left( 1-\alpha _{n}^{2}\right)
\delta \left\Vert p_{n}-u_{n}\right\Vert +\left( 1-\alpha _{n}^{1}\right)
\left( 1-\alpha _{n}^{2}\right) \left\Vert p_{n}-Tp_{n}\right\Vert  \notag \\
&&+\left( 1-\alpha _{n}^{1}\right) \alpha _{n}^{2}\left\Vert
p_{n}-Tp_{n}\right\Vert +\left( 1-\alpha _{n}^{1}\right) \alpha
_{n}^{2}\delta \left\Vert p_{n}-y_{n}\right\Vert  \notag \\
&&+\alpha _{n}^{1}\delta \left\Vert q_{n}-v_{n}\right\Vert +\left( 1-\alpha
_{n}^{1}\right) \left\Vert p_{n}-Tp_{n}\right\Vert  \notag \\
&=&\left( 1-\alpha _{n}^{1}\right) \left( 1-\alpha _{n}^{2}\right) \delta
\left\Vert p_{n}-u_{n}\right\Vert  \notag \\
&&+\left( 1-\alpha _{n}^{1}\right) \alpha _{n}^{2}\delta \left\Vert \left(
1-\alpha _{n}^{3}+\alpha _{n}^{3}\right) p_{n}-\left( 1-\alpha
_{n}^{3}\right) u_{n}-\alpha _{n}^{3}Tu_{n}\right\Vert  \notag \\
&&+\alpha _{n}^{1}\delta \left\Vert q_{n}-v_{n}\right\Vert +2\left( 1-\alpha
_{n}^{1}\right) \left\Vert p_{n}-Tp_{n}\right\Vert  \notag \\
&\leq &\left( 1-\alpha _{n}^{1}\right) \left( 1-\alpha _{n}^{2}\right)
\delta \left\Vert p_{n}-u_{n}\right\Vert  \notag \\
&&+\left( 1-\alpha _{n}^{1}\right) \alpha _{n}^{2}\delta \left( 1-\alpha
_{n}^{3}\right) \left\Vert p_{n}-u_{n}\right\Vert +\left( 1-\alpha
_{n}^{1}\right) \alpha _{n}^{2}\delta \alpha _{n}^{3}\left\Vert
p_{n}-Tu_{n}\right\Vert  \notag \\
&&+\alpha _{n}^{1}\delta \left\Vert q_{n}-v_{n}\right\Vert +2\left( 1-\alpha
_{n}^{1}\right) \left\Vert p_{n}-Tp_{n}\right\Vert  \notag \\
&\leq &\left( 1-\alpha _{n}^{1}\right) \left( 1-\alpha _{n}^{2}\right)
\delta \left\Vert p_{n}-u_{n}\right\Vert +\left( 1-\alpha _{n}^{1}\right)
\alpha _{n}^{2}\delta \left( 1-\alpha _{n}^{3}\right) \left\Vert
p_{n}-u_{n}\right\Vert  \notag \\
&&+\left( 1-\alpha _{n}^{1}\right) \alpha _{n}^{2}\delta \alpha
_{n}^{3}\left\Vert p_{n}-Tp_{n}\right\Vert +\left( 1-\alpha _{n}^{1}\right)
\alpha _{n}^{2}\delta \alpha _{n}^{3}\delta \left\Vert p_{n}-u_{n}\right\Vert
\notag \\
&&+\alpha _{n}^{1}\delta \left\Vert q_{n}-v_{n}\right\Vert +2\left( 1-\alpha
_{n}^{1}\right) \left\Vert p_{n}-Tp_{n}\right\Vert  \notag \\
&=&\left\{ \left( 1-\alpha _{n}^{1}\right) \left( 1-\alpha _{n}^{2}\right)
\delta +\left( 1-\alpha _{n}^{1}\right) \alpha _{n}^{2}\delta \left[
1-\alpha _{n}^{3}\left( 1-\delta \right) \right] \right\} \left\Vert
p_{n}-u_{n}\right\Vert  \notag \\
&&+\alpha _{n}^{1}\delta \left\Vert q_{n}-v_{n}\right\Vert +\left( 1-\alpha
_{n}^{1}\right) \left( 2+\alpha _{n}^{2}\delta \alpha _{n}^{3}\right)
\left\Vert p_{n}-Tp_{n}\right\Vert \text{,}  \label{eqn20}
\end{eqnarray}%
\begin{eqnarray}
\left\Vert q_{n}-v_{n}\right\Vert &=&\left\Vert \left( 1-\alpha
_{n}^{2}\right) Tp_{n}+\alpha _{n}^{2}Tr_{n}-\left( 1-\alpha _{n}^{2}\right)
Tu_{n}-\alpha _{n}^{2}Ty_{n}\right\Vert  \notag \\
&\leq &\left( 1-\alpha _{n}^{2}\right) \delta \left\Vert
p_{n}-u_{n}\right\Vert +\alpha _{n}^{2}\delta \left\Vert
r_{n}-y_{n}\right\Vert  \notag \\
&=&\left( 1-\alpha _{n}^{2}\right) \delta \left\Vert p_{n}-u_{n}\right\Vert 
\notag \\
&&+\alpha _{n}^{2}\delta \left\Vert \left( 1-\alpha _{n}^{3}\right)
p_{n}+\alpha _{n}^{3}Tp_{n}-\left( 1-\alpha _{n}^{3}\right) u_{n}-\alpha
_{n}^{3}Tu_{n}\right\Vert  \notag \\
&\leq &\left( 1-\alpha _{n}^{2}\right) \delta \left\Vert
p_{n}-u_{n}\right\Vert +\alpha _{n}^{2}\delta \left[ 1-\alpha _{n}^{3}\left(
1-\delta \right) \right] \left\Vert p_{n}-u_{n}\right\Vert  \notag \\
&\leq &\left( 1-\alpha _{n}^{2}\right) \left\Vert p_{n}-u_{n}\right\Vert
+\alpha _{n}^{2}\delta \left\Vert p_{n}-u_{n}\right\Vert  \notag \\
&=&\left[ 1-\alpha _{n}^{2}\left( 1-\delta \right) \right] \left\Vert
p_{n}-u_{n}\right\Vert \text{.}  \label{eqn21}
\end{eqnarray}%
Substituting (2.9) in (2.8) 
\begin{eqnarray}
\left\Vert p_{n+1}-u_{n+1}\right\Vert &\leq &\left\{ \left( 1-\alpha
_{n}^{1}\right) \left( 1-\alpha _{n}^{2}\right) \delta +\left( 1-\alpha
_{n}^{1}\right) \alpha _{n}^{2}\delta \left[ 1-\alpha _{n}^{3}\left(
1-\delta \right) \right] \right.  \notag \\
&&\left. +\alpha _{n}^{1}\delta \left[ 1-\alpha _{n}^{2}\left( 1-\delta
\right) \right] \right\} \left\Vert p_{n}-u_{n}\right\Vert  \notag \\
&&+\left( 1-\alpha _{n}^{1}\right) \left( 2+\alpha _{n}^{2}\delta \alpha
_{n}^{3}\right) \left\Vert p_{n}-Tp_{n}\right\Vert \text{.}  \label{eqn22}
\end{eqnarray}%
Since $\delta \in \left[ 0,1\right) $, $\alpha _{n}^{i}\in \left[ 0,1\right] 
$ for all $n\in 
\mathbb{N}
$ and for each $i\in \left\{ 1,2,3\right\} $,%
\begin{equation}
1-\alpha _{n}^{2}\left( 1-\delta \right) <1\text{,}  \label{eqn23}
\end{equation}%
\begin{equation}
1-\alpha _{n}^{3}\left( 1-\delta \right) <1\text{.}  \label{eqn24}
\end{equation}%
By applying inequalities (2.11) and (2.12) to (2.10), we obtain%
\begin{eqnarray}
\left\Vert p_{n+1}-u_{n+1}\right\Vert &\leq &\left[ 1-\alpha _{n}^{1}\left(
1-\delta \right) \right] \left\Vert p_{n}-u_{n}\right\Vert  \notag \\
&&+\left( 1-\alpha _{n}^{1}\right) \left( 2+\alpha _{n}^{2}\delta \alpha
_{n}^{3}\right) \left\Vert p_{n}-Tp_{n}\right\Vert \text{.}  \label{eqn25}
\end{eqnarray}%
Using the fact $x_{\ast }=Tx_{\ast }$ and triangle inequality for norms, we
derive%
\begin{eqnarray}
\left\Vert p_{n}-Tp_{n}\right\Vert &=&\left\Vert p_{n}-x_{\ast }+Tx_{\ast
}-Tp_{n}\right\Vert  \notag \\
&\leq &\left\Vert p_{n}-x_{\ast }\right\Vert +\left\Vert Tx_{\ast
}-Tp_{n}\right\Vert  \notag \\
&\leq &\left( 1+\delta \right) \left\Vert p_{n}-x_{\ast }\right\Vert \text{.}
\label{eqn26}
\end{eqnarray}%
Substituting (2.14) in (2.13)%
\begin{eqnarray}
\left\Vert p_{n+1}-u_{n+1}\right\Vert &\leq &\left[ 1-\alpha _{n}^{1}\left(
1-\delta \right) \right] \left\Vert p_{n}-u_{n}\right\Vert  \notag \\
&&+\left( 1-\alpha _{n}^{1}\right) \left( 2+\alpha _{n}^{2}\delta \alpha
_{n}^{3}\right) \left( 1+\delta \right) \left\Vert p_{n}-x_{\ast
}\right\Vert \text{.}  \label{eqn27}
\end{eqnarray}%
Denote that%
\begin{eqnarray}
a_{n} &=&\left\Vert p_{n}-u_{n}\right\Vert \text{,}  \notag \\
\eta _{n} &=&\alpha _{n}^{1}\left( 1-\delta \right) \in \left( 0,1\right) 
\text{,}  \label{eqn28} \\
\rho _{n} &=&\left( 1-\alpha _{n}^{1}\right) \left( 2+\alpha _{n}^{2}\delta
\alpha _{n}^{3}\right) \left( 1+\delta \right) \left\Vert p_{n}-x_{\ast
}\right\Vert \text{.}  \notag
\end{eqnarray}%
Thus, an application of Lemma 1 to (2.15) yields $a_{n}=\left\Vert
p_{n}-u_{n}\right\Vert \rightarrow 0$ as $n\rightarrow \infty $. Also, since 
$\left\Vert u_{n}-x_{\ast }\right\Vert \leq \left\Vert
p_{n}-u_{n}\right\Vert +\left\Vert p_{n}-x_{\ast }\right\Vert $, we have $%
\left\Vert u_{n}-x_{\ast }\right\Vert \rightarrow 0$ as $n\rightarrow \infty 
$.

Next, we will prove (ii)$\Rightarrow $(i). \ Assume that $\left\Vert
u_{n}-x_{\ast }\right\Vert \rightarrow 0$ as $n\rightarrow \infty $. It
follows from CR iteration method (1.6), iteration method (1.7) and condition
(1.2) that%
\begin{eqnarray}
\left\Vert u_{n+1}-p_{n+1}\right\Vert &=&\left\Vert \left( 1-\alpha
_{n}^{1}\right) v_{n}+\alpha _{n}^{1}Tv_{n}-\left( 1-\alpha _{n}^{1}\right)
Tp_{n}-\alpha _{n}^{1}Tq_{n}\right\Vert  \notag \\
&\leq &\left( 1-\alpha _{n}^{1}\right) \left\Vert v_{n}-Tp_{n}\right\Vert
+\alpha _{n}^{1}\left\Vert Tv_{n}-Tq_{n}\right\Vert  \notag \\
&\leq &\left( 1-\alpha _{n}^{1}\right) \left\Vert \left( 1-\alpha
_{n}^{2}\right) Tu_{n}+\alpha _{n}^{2}Ty_{n}-Tp_{n}\right\Vert +\alpha
_{n}^{1}\delta \left\Vert v_{n}-q_{n}\right\Vert  \notag \\
&=&\left( 1-\alpha _{n}^{1}\right) \left\Vert \left( 1-\alpha
_{n}^{2}\right) Tu_{n}+\alpha _{n}^{2}Ty_{n}-\left( 1-\alpha _{n}^{2}+\alpha
_{n}^{2}\right) Tp_{n}\right\Vert  \notag \\
&&+\alpha _{n}^{1}\delta \left\Vert \left( 1-\alpha _{n}^{2}\right)
Tu_{n}+\alpha _{n}^{2}Ty_{n}-\left( 1-\alpha _{n}^{2}\right) Tp_{n}-\alpha
_{n}^{2}Tr_{n}\right\Vert  \notag \\
&\leq &\left( 1-\alpha _{n}^{1}\right) \left( 1-\alpha _{n}^{2}\right)
\left\Vert Tu_{n}-Tp_{n}\right\Vert +\left( 1-\alpha _{n}^{1}\right) \alpha
_{n}^{2}\left\Vert Ty_{n}-Tp_{n}\right\Vert  \notag \\
&&+\alpha _{n}^{1}\delta \left( 1-\alpha _{n}^{2}\right) \left\Vert
Tu_{n}-Tp_{n}\right\Vert +\alpha _{n}^{1}\delta \alpha _{n}^{2}\left\Vert
Ty_{n}-Tr_{n}\right\Vert  \notag \\
&\leq &\left( 1-\alpha _{n}^{1}\right) \left( 1-\alpha _{n}^{2}\right)
\delta \left\Vert u_{n}-p_{n}\right\Vert +\left( 1-\alpha _{n}^{1}\right)
\alpha _{n}^{2}\delta \left\Vert y_{n}-p_{n}\right\Vert  \notag \\
&&+\alpha _{n}^{1}\delta \left( 1-\alpha _{n}^{2}\right) \delta \left\Vert
u_{n}-p_{n}\right\Vert +\alpha _{n}^{1}\delta \alpha _{n}^{2}\delta
\left\Vert y_{n}-r_{n}\right\Vert  \notag \\
&=&\left\{ \left( 1-\alpha _{n}^{1}\right) \left( 1-\alpha _{n}^{2}\right)
\delta +\alpha _{n}^{1}\delta \left( 1-\alpha _{n}^{2}\right) \delta
\right\} \left\Vert u_{n}-p_{n}\right\Vert  \notag \\
&&+\left( 1-\alpha _{n}^{1}\right) \alpha _{n}^{2}\delta \left\Vert \left(
1-\alpha _{n}^{3}\right) u_{n}+\alpha _{n}^{3}Tu_{n}-\left( 1-\alpha
_{n}^{3}+\alpha _{n}^{3}\right) p_{n}\right\Vert  \notag \\
&&+\alpha _{n}^{1}\delta \alpha _{n}^{2}\delta \left\Vert \left( 1-\alpha
_{n}^{3}\right) u_{n}+\alpha _{n}^{3}Tu_{n}-\left( 1-\alpha _{n}^{3}\right)
p_{n}-\alpha _{n}^{3}Tp_{n}\right\Vert  \notag \\
&\leq &\left\{ \left( 1-\alpha _{n}^{1}\right) \left( 1-\alpha
_{n}^{2}\right) \delta +\alpha _{n}^{1}\delta \left( 1-\alpha
_{n}^{2}\right) \delta \right\} \left\Vert u_{n}-p_{n}\right\Vert  \notag \\
&&+\left( 1-\alpha _{n}^{1}\right) \alpha _{n}^{2}\delta \left( 1-\alpha
_{n}^{3}\right) \left\Vert u_{n}-p_{n}\right\Vert +\left( 1-\alpha
_{n}^{1}\right) \alpha _{n}^{2}\delta \alpha _{n}^{3}\left\Vert
Tu_{n}-p_{n}\right\Vert  \notag \\
&&+\alpha _{n}^{1}\delta \alpha _{n}^{2}\delta \left( 1-\alpha
_{n}^{3}\right) \left\Vert u_{n}-p_{n}\right\Vert +\alpha _{n}^{1}\delta
\alpha _{n}^{2}\delta \alpha _{n}^{3}\left\Vert Tu_{n}-Tp_{n}\right\Vert 
\notag \\
&\leq &\left\{ \left( 1-\alpha _{n}^{1}\right) \left( 1-\alpha
_{n}^{2}\right) \delta +\alpha _{n}^{1}\delta \left( 1-\alpha
_{n}^{2}\right) \delta \right.  \notag \\
&&\left. +\left( 1-\alpha _{n}^{1}\right) \alpha _{n}^{2}\delta \left(
1-\alpha _{n}^{3}\right) +\left( 1-\alpha _{n}^{1}\right) \alpha
_{n}^{2}\delta \alpha _{n}^{3}\right.  \notag \\
&&\left. +\alpha _{n}^{1}\delta \alpha _{n}^{2}\delta \left( 1-\alpha
_{n}^{3}\right) +\alpha _{n}^{1}\delta \alpha _{n}^{2}\delta \alpha
_{n}^{3}\delta \right\} \left\Vert u_{n}-p_{n}\right\Vert  \notag \\
&&+\left( 1-\alpha _{n}^{1}\right) \alpha _{n}^{2}\delta \alpha
_{n}^{3}\left\Vert u_{n}-Tu_{n}\right\Vert  \notag \\
&=&\left\{ \left( 1-\alpha _{n}^{1}\right) \left( 1-\alpha _{n}^{2}\right)
\delta +\alpha _{n}^{1}\delta \left( 1-\alpha _{n}^{2}\right) \delta \right.
\notag \\
&&\left. +\left( 1-\alpha _{n}^{1}\right) \alpha _{n}^{2}\delta \left(
1-\alpha _{n}^{3}\right) +\left( 1-\alpha _{n}^{1}\right) \alpha
_{n}^{2}\delta \alpha _{n}^{3}\right.  \notag \\
&&\left. +\alpha _{n}^{1}\delta \alpha _{n}^{2}\delta \left[ 1-\alpha
_{n}^{3}\left( 1-\delta \right) \right] \right\} \left\Vert
u_{n}-p_{n}\right\Vert  \notag \\
&&+\left( 1-\alpha _{n}^{1}\right) \alpha _{n}^{2}\delta \alpha
_{n}^{3}\left\Vert u_{n}-Tu_{n}\right\Vert \text{.}  \label{eqn29}
\end{eqnarray}%
or,%
\begin{eqnarray}
\left\Vert u_{n+1}-p_{n+1}\right\Vert &\leq &\left\{ \left[ 1-\alpha
_{n}^{1}\left( 1-\delta \right) \right] \left( 1-\alpha _{n}^{2}\right)
\delta +\left( 1-\alpha _{n}^{1}\right) \alpha _{n}^{2}\delta \right.  \notag
\\
&&\left. +\alpha _{n}^{1}\delta \alpha _{n}^{2}\delta \left[ 1-\alpha
_{n}^{3}\left( 1-\delta \right) \right] \right\} \left\Vert
u_{n}-p_{n}\right\Vert  \notag \\
&&+\left( 1-\alpha _{n}^{1}\right) \alpha _{n}^{2}\delta \alpha
_{n}^{3}\left\Vert u_{n}-Tu_{n}\right\Vert \text{.}  \label{eqn30}
\end{eqnarray}%
Since $\delta \in \left[ 0,1\right) $, $\alpha _{n}^{i}\in \left[ 0,1\right] 
$ for all $n\in 
\mathbb{N}
$ and for each $i\in \left\{ 1,2,3\right\} $,%
\begin{equation}
1-\alpha _{n}^{2}\left( 1-\delta \right) <1\text{,}  \label{eqn31}
\end{equation}%
\begin{equation}
1-\alpha _{n}^{3}\left( 1-\delta \right) <1\text{.}  \label{eqn32}
\end{equation}%
By use of inequalities (2.19) and (2.20) in (2.18), we get%
\begin{eqnarray}
\left\Vert u_{n+1}-p_{n+1}\right\Vert &\leq &\left[ 1-\alpha _{n}^{1}\left(
1-\delta \right) \right] \left\Vert u_{n}-p_{n}\right\Vert  \notag \\
&&+\left( 1-\alpha _{n}^{1}\right) \alpha _{n}^{2}\delta \alpha
_{n}^{3}\left\Vert u_{n}-Tu_{n}\right\Vert \text{.}  \label{eqn33}
\end{eqnarray}%
Using the fact $x_{\ast }=Tx_{\ast }$ and triangle inequality for norms, we
derive%
\begin{equation}
\left\Vert u_{n}-Tu_{n}\right\Vert \leq \left( 1+\delta \right) \left\Vert
u_{n}-x_{\ast }\right\Vert \text{.}  \label{eqn34}
\end{equation}%
Hence, (2.21) becomes%
\begin{eqnarray}
\left\Vert u_{n+1}-p_{n+1}\right\Vert &\leq &\left[ 1-\alpha _{n}^{1}\left(
1-\delta \right) \right] \left\Vert u_{n}-p_{n}\right\Vert  \notag \\
&&+\left( 1-\alpha _{n}^{1}\right) \alpha _{n}^{2}\delta \alpha
_{n}^{3}\left( 1+\delta \right) \left\Vert u_{n}-x_{\ast }\right\Vert \text{.%
}  \label{eqn35}
\end{eqnarray}%
Define%
\begin{eqnarray}
a_{n} &=&\left\Vert u_{n}-p_{n}\right\Vert \text{,}  \notag \\
\eta _{n} &=&\alpha _{n}^{1}\left( 1-\delta \right) \in \left( 0,1\right) 
\text{,}  \label{eqn36} \\
\rho _{n} &=&\left( 1-\alpha _{n}^{1}\right) \alpha _{n}^{2}\delta \alpha
_{n}^{3}\left( 1+\delta \right) \left\Vert u_{n}-x_{\ast }\right\Vert \text{.%
}  \notag
\end{eqnarray}%
Thus, an application of Lemma 1 to (2.23) yields $a_{n}=\left\Vert
u_{n}-p_{n}\right\Vert \rightarrow 0$ as $n\rightarrow \infty $. Also, since 
$\left\Vert p_{n}-x_{\ast }\right\Vert \leq \left\Vert
p_{n}-u_{n}\right\Vert +\left\Vert u_{n}-x_{\ast }\right\Vert $, we have $%
\left\Vert p_{n}-x_{\ast }\right\Vert \rightarrow 0$ as $n\rightarrow \infty 
$.
\end{proof}

\begin{theorem}
Let $S$, $B$ and $T$ with fixed point $x_{\ast }$ be as in Theorem 1. Let $%
\left\{ \alpha _{n}^{i}\right\} _{n=0}^{\infty }$, $i\in \left\{
1,2,3\right\} $ be real sequences in $\left[ 0,1\right] $ satisfying (i) $%
\alpha _{1}\leq \alpha _{n}^{1}\leq 1$, $\alpha _{2}\leq \alpha _{n}^{2}\leq
1$ and $\alpha _{3}\leq \alpha _{n}^{3}\leq 1$, for all $n\in 
\mathbb{N}
$ and for some $\alpha _{1}$, $\alpha _{2}$, $\alpha _{3}>0$. For given $%
u_{0}=p_{0}\in S$, consider iterative sequences $\left\{ u_{n}\right\}
_{n=0}^{\infty }$ and $\left\{ p_{n}\right\} _{n=0}^{\infty }$ defined by
(1.6) and (1.7), respectively. Then $\left\{ p_{n}\right\} _{n=0}^{\infty }$
converges to $x_{\ast }$ faster than $\left\{ u_{n}\right\} _{n=0}^{\infty }$
does.
\end{theorem}

\begin{proof}
The following equality was obtained in (\cite{Karahan}, Theorem 1) 
\begin{equation}
b_{n}=\left\Vert p_{0}-x_{\ast }\right\Vert \delta ^{n+1}\left[ 1-\alpha
_{1}\left( 1-\delta \left( 1-\alpha _{2}\alpha _{3}\left( 1-\delta \right)
\right) \right) \right] ^{n+1}\text{.}  \label{eqn37}
\end{equation}%
Using now (1.7) and (1.2) we have%
\begin{eqnarray}
\left\Vert u_{n+1}-x_{\ast }\right\Vert &=&\left\Vert \left( 1-\alpha
_{n}^{1}\right) v_{n}+\alpha _{n}^{1}Tv_{n}-x_{\ast }\right\Vert  \notag \\
&\leq &\left( 1-\alpha _{n}^{1}\right) \left\Vert v_{n}-x_{\ast }\right\Vert
+\alpha _{n}^{1}\left\Vert Tv_{n}-x_{\ast }\right\Vert  \notag \\
&\leq &\left[ \left( 1-\alpha _{n}^{1}\right) +\alpha _{n}^{1}\delta \right]
\left\Vert v_{n}-x_{\ast }\right\Vert  \notag \\
&\leq &\left[ \left( 1-\alpha _{n}^{1}\right) +\alpha _{n}^{1}\delta \right]
\left\Vert \left( 1-\alpha _{n}^{2}\right) Tu_{n}+\alpha
_{n}^{2}Ty_{n}-x_{\ast }\right\Vert  \notag \\
&\leq &\left[ \left( 1-\alpha _{n}^{1}\right) +\alpha _{n}^{1}\delta \right]
\left( 1-\alpha _{n}^{2}\right) \left\Vert Tu_{n}-x_{\ast }\right\Vert 
\notag \\
&&+\left[ \left( 1-\alpha _{n}^{1}\right) +\alpha _{n}^{1}\delta \right]
\alpha _{n}^{2}\left\Vert Ty_{n}-x_{\ast }\right\Vert  \notag \\
&\leq &\left[ \left( 1-\alpha _{n}\right) +\alpha _{n}^{1}\delta \right]
\left( 1-\alpha _{n}^{2}\right) \delta \left\Vert u_{n}-x_{\ast }\right\Vert
\notag \\
&&+\left[ \left( 1-\alpha _{n}^{1}\right) +\alpha _{n}^{1}\delta \right]
\alpha _{n}^{2}\delta \left\Vert y_{n}-x_{\ast }\right\Vert  \notag \\
&\leq &\left[ \left( 1-\alpha _{n}^{1}\right) +\alpha _{n}^{1}\delta \right]
\left( 1-\alpha _{n}^{2}\right) \delta \left\Vert u_{n}-x_{\ast }\right\Vert
\notag \\
&&+\left[ \left( 1-\alpha _{n}^{1}\right) +\alpha _{n}^{1}\delta \right]
\alpha _{n}^{2}\delta \left( 1-\alpha _{n}^{3}\right) \left\Vert
u_{n}-x_{\ast }\right\Vert  \notag \\
&&+\left[ \left( 1-\alpha _{n}^{1}\right) +\alpha _{n}^{1}\delta \right]
\alpha _{n}^{2}\delta \alpha _{n}^{3}\delta \left\Vert u_{n}-x_{\ast
}\right\Vert  \notag \\
&=&\left[ \left( 1-\alpha _{n}^{1}\right) +\alpha _{n}^{1}\delta \right]
\left\{ \left( 1-\alpha _{n}^{2}\right) \delta +\alpha _{n}^{2}\delta \left(
1-\alpha _{n}^{3}\right) +\alpha _{n}^{2}\delta \alpha _{n}^{3}\delta
\right\} \left\Vert u_{n}-x_{\ast }\right\Vert  \notag \\
&=&\left[ 1-\alpha _{n}^{1}\left( 1-\delta \right) \right] \left[ 1-\alpha
_{n}^{2}\alpha _{n}^{3}\left( 1-\delta \right) \right] \delta \left\Vert
u_{n}-x_{\ast }\right\Vert  \notag \\
&\leq &\cdots  \notag \\
&\leq &\prod\limits_{k=0}^{n}\left[ 1-\alpha _{k}^{1}\left( 1-\delta \right) %
\right] \left[ 1-\alpha _{k}^{2}\alpha _{k}^{3}\left( 1-\delta \right) %
\right] \delta \left\Vert u_{0}-x_{\ast }\right\Vert \text{.}  \label{eqn38}
\end{eqnarray}%
From assumption (i), we obtain%
\begin{equation}
\left\Vert u_{n+1}-x_{\ast }\right\Vert \leq \left\Vert u_{0}-x_{\ast
}\right\Vert \delta ^{n+1}\left[ 1-\alpha _{1}\left( 1-\delta \right) \right]
^{n+1}\left[ 1-\alpha _{2}\alpha _{3}\left( 1-\delta \right) \right] ^{n+1}%
\text{.}  \label{eqn39}
\end{equation}%
Let%
\begin{equation}
a_{n}=\left\Vert u_{0}-x_{\ast }\right\Vert \delta ^{n+1}\left[ 1-\alpha
_{1}\left( 1-\delta \right) \right] ^{n+1}\left[ 1-\alpha _{2}\alpha
_{3}\left( 1-\delta \right) \right] ^{n+1}\text{.}  \label{eqn40}
\end{equation}%
Define%
\begin{eqnarray}
\theta _{n} &=&\frac{a_{n}}{b_{n}}=\frac{\left\Vert u_{0}-x_{\ast
}\right\Vert \delta ^{n+1}\left[ 1-\alpha _{1}\left( 1-\delta \right) \right]
^{n+1}\left[ 1-\alpha _{2}\alpha _{3}\left( 1-\delta \right) \right] ^{n+1}}{%
\left\Vert p_{0}-x_{\ast }\right\Vert \delta ^{n+1}\left[ 1-\alpha
_{1}\left( 1-\delta \left( 1-\alpha _{2}\alpha _{3}\left( 1-\delta \right)
\right) \right) \right] ^{n+1}}  \notag \\
&=&\frac{\left[ 1-\alpha _{1}\left( 1-\delta \right) \right] ^{n+1}\left[
1-\alpha _{2}\alpha _{3}\left( 1-\delta \right) \right] ^{n+1}}{\left[
1-\alpha _{1}\left( 1-\delta \left( 1-\alpha _{2}\alpha _{3}\left( 1-\delta
\right) \right) \right) \right] ^{n+1}}  \label{eqn41}
\end{eqnarray}%
Since $\delta \in \left( 0,1\right) $ and $\alpha _{i}\in \left( 0,1\right) $
for each $i\in \left\{ 1,2,3\right\} $%
\begin{eqnarray}
\alpha _{1} &<&1  \notag \\
&\Rightarrow &\alpha _{1}\alpha _{2}\alpha _{3}\left( 1-\delta \right)
<\alpha _{2}\alpha _{3}\left( 1-\delta \right)  \notag \\
&\Rightarrow &\alpha _{2}\alpha _{3}\left( -1+\delta \right) +\alpha
_{1}\alpha _{2}\alpha _{3}\left( 1-\delta \right) <0  \notag \\
&\Rightarrow &-\alpha _{2}\alpha _{3}+\alpha _{2}\alpha _{3}\delta +\alpha
_{1}\alpha _{2}\alpha _{3}-\alpha _{1}\alpha _{2}\alpha _{3}\delta <0  \notag
\\
&\Rightarrow &-\alpha _{2}\alpha _{3}+\alpha _{2}\alpha _{3}\delta +\alpha
_{1}\alpha _{2}\alpha _{3}-2\alpha _{1}\alpha _{2}\alpha _{3}\delta <-\alpha
_{1}\alpha _{2}\alpha _{3}\delta  \notag \\
&\Rightarrow &\left\{ 
\begin{array}{c}
1-\alpha _{1}+\alpha _{1}\delta -\alpha _{2}\alpha _{3}+\alpha _{2}\alpha
_{3}\delta +\alpha _{1}\alpha _{2}\alpha _{3}-2\alpha _{1}\alpha _{2}\alpha
_{3}\delta +\alpha _{1}\alpha _{2}\alpha _{3}\delta ^{2} \\ 
<1-\alpha _{1}+\alpha _{1}\delta -\alpha _{1}\alpha _{2}\alpha _{3}\delta
+\alpha _{1}\alpha _{2}\alpha _{3}\delta ^{2}%
\end{array}%
\right.  \notag \\
&\Rightarrow &\left\{ 
\begin{array}{c}
1-\alpha _{1}+\alpha _{1}\delta -\alpha _{2}\alpha _{3}\left( 1-\delta
\right) +\alpha _{1}\alpha _{2}\alpha _{3}\left( 1-\delta \right) ^{2} \\ 
<1-\alpha _{1}+\alpha _{1}\delta -\alpha _{1}\alpha _{2}\alpha _{3}\delta
\left( 1-\delta \right)%
\end{array}%
\right.  \notag \\
&\Rightarrow &\left\{ 
\begin{array}{c}
1-\alpha _{1}\left( 1-\delta \right) -\left[ 1-\alpha _{1}\left( 1-\delta
\right) \right] \alpha _{2}\alpha _{3}\left( 1-\delta \right) \\ 
<1-\alpha _{1}+\alpha _{1}\delta \left( 1-\alpha _{2}\alpha _{3}\left(
1-\delta \right) \right)%
\end{array}%
\right.  \notag \\
&\Rightarrow &\left[ 1-\alpha _{1}\left( 1-\delta \right) \right] \left[
1-\alpha _{2}\alpha _{3}\left( 1-\delta \right) \right] <1-\alpha _{1}\left(
1-\delta \left( 1-\alpha _{2}\alpha _{3}\left( 1-\delta \right) \right)
\right)  \notag \\
&\Rightarrow &\frac{\left[ 1-\alpha _{1}\left( 1-\delta \right) \right] %
\left[ 1-\alpha _{2}\alpha _{3}\left( 1-\delta \right) \right] }{1-\alpha
_{1}\left( 1-\delta \left( 1-\alpha _{2}\alpha _{3}\left( 1-\delta \right)
\right) \right) }<1\text{,}  \label{eqn42}
\end{eqnarray}%
and thus, we have%
\begin{eqnarray}
\lim_{n\rightarrow \infty }\frac{\theta _{n+1}}{\theta _{n}}
&=&\lim_{n\rightarrow \infty }\frac{\frac{\left[ 1-\alpha _{1}\left(
1-\delta \right) \right] ^{n+2}\left[ 1-\alpha _{2}\alpha _{3}\left(
1-\delta \right) \right] ^{n+2}}{\left[ 1-\alpha _{1}\left( 1-\delta \left(
1-\alpha _{2}\alpha _{3}\left( 1-\delta \right) \right) \right) \right]
^{n+2}}}{\frac{\left[ 1-\alpha _{1}\left( 1-\delta \right) \right] ^{n+1}%
\left[ 1-\alpha _{2}\alpha _{3}\left( 1-\delta \right) \right] ^{n+1}}{\left[
1-\alpha _{1}\left( 1-\delta \left( 1-\alpha _{2}\alpha _{3}\left( 1-\delta
\right) \right) \right) \right] ^{n+1}}}  \notag \\
&=&\lim_{n\rightarrow \infty }\frac{\left[ 1-\alpha _{1}\left( 1-\delta
\right) \right] \left[ 1-\alpha _{2}\alpha _{3}\left( 1-\delta \right) %
\right] }{1-\alpha _{1}\left( 1-\delta \left( 1-\alpha _{2}\alpha _{3}\left(
1-\delta \right) \right) \right) }  \notag \\
&=&\frac{\left[ 1-\alpha _{1}\left( 1-\delta \right) \right] \left[ 1-\alpha
_{2}\alpha _{3}\left( 1-\delta \right) \right] }{1-\alpha _{1}\left(
1-\delta \left( 1-\alpha _{2}\alpha _{3}\left( 1-\delta \right) \right)
\right) }<1\text{.}  \label{eqn43}
\end{eqnarray}%
It thus follows from ratio test that $\sum\limits_{n=0}^{\infty }\theta
_{n}<\infty $. Hence, we have $\lim_{n\rightarrow \infty }\theta _{n}=0$
which implies that $\left\{ p_{n}\right\} _{n=0}^{\infty }$ is faster than $%
\left\{ u_{n}\right\} _{n=0}^{\infty }$.
\end{proof}

We are now able to establish the following data dependence result.

\begin{theorem}
Let $\widetilde{T}$ be an approximate operator of $T$ satisfying condition
(1.2). Let $\left\{ p_{n}\right\} _{n=0}^{\infty }$ be an iterative sequence
generated by (1.7) for $T$ and define an iterative sequence $\left\{ 
\widetilde{p}_{n}\right\} _{n=0}^{\infty }$ as follows%
\begin{equation}
\left\{ 
\begin{array}{c}
\widetilde{p}_{0}\in S\text{, \ \ \ \ \ \ \ \ \ \ \ \ \ \ \ \ \ \ \ \ \ \ \
\ \ \ \ \ \ \ \ \ \ \ \ \ \ \ \ \ } \\ 
\widetilde{p}_{n+1}=\left( 1-\alpha _{n}^{1}\right) \widetilde{T}\widetilde{p%
}_{n}+\alpha _{n}^{1}\widetilde{T}\widetilde{q}_{n}\text{, \ \ \ \ \ \ \ \ \
\ } \\ 
\widetilde{q}_{n}=\left( 1-\alpha _{n}^{2}\right) \widetilde{T}\widetilde{p}%
_{n}+\alpha _{n}^{2}\widetilde{T}\widetilde{r}_{n}\text{, \ \ \ \ \ \ } \\ 
\widetilde{r}_{n}=\left( 1-\alpha _{n}^{3}\right) \widetilde{p}_{n}+\alpha
_{n}^{3}\widetilde{T}\widetilde{p}_{n}\text{, }n\in 
\mathbb{N}
\text{,}%
\end{array}%
\right.  \label{eqn44}
\end{equation}%
where $\left\{ \alpha _{n}^{i}\right\} _{n=0}^{\infty }$, $i\in \left\{
1,2,3\right\} $ be real sequences in $\left[ 0,1\right] $ satisfying (i) $%
\frac{1}{2}\leq \alpha _{n}^{1}$ for all $n\in 
\mathbb{N}
$, and (ii) $\sum\limits_{n=0}^{\infty }\alpha _{n}^{1}=\infty $. If $Tp=p$
and $\widetilde{T}\widetilde{p}=\widetilde{p}$ such that $\widetilde{p}%
_{n}\rightarrow \widetilde{p}$ as $n\rightarrow \infty $, then we have%
\begin{equation}
\left\Vert p-\widetilde{p}\right\Vert \leq \frac{5\varepsilon }{1-\delta }%
\text{,}  \label{eqn45}
\end{equation}%
where $\varepsilon >0$ is a fixed number.
\end{theorem}

\begin{proof}
It follows from (1.2), (1.7), and (2.32) that%
\begin{eqnarray}
\left\Vert r_{n}-\widetilde{r}_{n}\right\Vert &=&\left\Vert \left( 1-\alpha
_{n}^{3}\right) p_{n}+\alpha _{n}^{3}Tp_{n}-\left( 1-\alpha _{n}^{3}\right) 
\widetilde{p}_{n}-\alpha _{n}^{3}\widetilde{T}\widetilde{p}_{n}\right\Vert 
\notag \\
&\leq &\left( 1-\alpha _{n}^{3}\right) \left\Vert p_{n}-\widetilde{p}%
_{n}\right\Vert +\alpha _{n}^{3}\left\Vert Tp_{n}-\widetilde{T}\widetilde{p}%
_{n}\right\Vert  \notag \\
&\leq &\left( 1-\alpha _{n}^{3}\right) \left\Vert p_{n}-\widetilde{p}%
_{n}\right\Vert +\alpha _{n}^{3}\left\{ \left\Vert Tp_{n}-T\widetilde{p}%
_{n}\right\Vert +\left\Vert T\widetilde{p}_{n}-\widetilde{T}\widetilde{p}%
_{n}\right\Vert \right\}  \notag \\
&\leq &\left[ 1-\alpha _{n}^{3}\left( 1-\delta \right) \right] \left\Vert
p_{n}-\widetilde{p}_{n}\right\Vert +\alpha _{n}^{3}\varepsilon \text{,}
\label{eqn46}
\end{eqnarray}%
\begin{eqnarray}
\left\Vert q_{n}-\widetilde{q}_{n}\right\Vert &=&\left\Vert \left( 1-\alpha
_{n}^{2}\right) Tp_{n}+\alpha _{n}^{2}Tr_{n}-\left( 1-\alpha _{n}^{2}\right) 
\widetilde{T}\widetilde{p}_{n}-\alpha _{n}^{2}\widetilde{T}\widetilde{r}%
_{n}\right\Vert  \notag \\
&\leq &\left( 1-\alpha _{n}^{2}\right) \left\Vert Tp_{n}-\widetilde{T}%
\widetilde{p}_{n}\right\Vert +\alpha _{n}^{2}\left\Vert Tr_{n}-\widetilde{T}%
\widetilde{r}_{n}\right\Vert  \notag \\
&\leq &\left( 1-\alpha _{n}^{2}\right) \left\{ \left\Vert Tp_{n}-T\widetilde{%
p}_{n}\right\Vert +\left\Vert T\widetilde{p}_{n}-\widetilde{T}\widetilde{p}%
_{n}\right\Vert \right\}  \notag \\
&&+\alpha _{n}^{2}\left\{ \left\Vert Tr_{n}-T\widetilde{r}_{n}\right\Vert
+\left\Vert T\widetilde{r}_{n}-\widetilde{T}\widetilde{r}_{n}\right\Vert
\right\}  \notag \\
&\leq &\left( 1-\alpha _{n}^{2}\right) \delta \left\Vert p_{n}-\widetilde{p}%
_{n}\right\Vert +\alpha _{n}^{2}\delta \left\Vert r_{n}-\widetilde{r}%
_{n}\right\Vert +\left( 1-\alpha _{n}^{2}\right) \varepsilon +\alpha
_{n}^{2}\varepsilon \text{,}  \label{eqn47}
\end{eqnarray}%
\begin{eqnarray}
\left\Vert p_{n+1}-\widetilde{p}_{n+1}\right\Vert &=&\left\Vert \left(
1-\alpha _{n}^{1}\right) Tp_{n}+\alpha _{n}^{1}Tq_{n}-\left( 1-\alpha
_{n}^{1}\right) \widetilde{T}\widetilde{p}_{n}-\alpha _{n}^{1}\widetilde{T}%
q_{n}\right\Vert  \notag \\
&\leq &\left( 1-\alpha _{n}^{1}\right) \left\Vert Tp_{n}-\widetilde{T}%
\widetilde{p}_{n}\right\Vert +\alpha _{n}^{1}\left\Vert Tq_{n}-\widetilde{T}%
\widetilde{q}_{n}\right\Vert  \notag \\
&\leq &\left( 1-\alpha _{n}^{1}\right) \left\{ \left\Vert Tp_{n}-T\widetilde{%
p}_{n}\right\Vert +\left\Vert T\widetilde{p}_{n}-\widetilde{T}\widetilde{p}%
_{n}\right\Vert \right\}  \notag \\
&&+\alpha _{n}^{1}\left\{ \left\Vert Tq_{n}-T\widetilde{q}_{n}\right\Vert
+\left\Vert T\widetilde{q}_{n}-\widetilde{T}\widetilde{q}_{n}\right\Vert
\right\}  \notag \\
&\leq &\left( 1-\alpha _{n}^{1}\right) \left\{ \delta \left\Vert p_{n}-%
\widetilde{p}_{n}\right\Vert +\varepsilon \right\} +\alpha _{n}^{1}\left\{
\delta \left\Vert q_{n}-\widetilde{q}_{n}\right\Vert +\varepsilon \right\} 
\notag \\
&=&\left( 1-\alpha _{n}^{1}\right) \delta \left\Vert p_{n}-\widetilde{p}%
_{n}\right\Vert +\alpha _{n}^{1}\delta \left\Vert q_{n}-\widetilde{q}%
_{n}\right\Vert +\left( 1-\alpha _{n}^{1}\right) \varepsilon +\alpha
_{n}^{1}\varepsilon \text{.}  \label{eqn48}
\end{eqnarray}%
Combining (2.34), (2.35), and (2.36)%
\begin{eqnarray}
\left\Vert p_{n+1}-\widetilde{p}_{n+1}\right\Vert &\leq &\left\{ \left(
1-\alpha _{n}^{1}\right) \delta +\alpha _{n}^{1}\delta \left\{ \left(
1-\alpha _{n}^{2}\right) \delta +\alpha _{n}^{2}\delta \left[ 1-\alpha
_{n}^{3}\left( 1-\delta \right) \right] \right\} \right\} \left\Vert p_{n}-%
\widetilde{p}_{n}\right\Vert  \notag \\
&&+\alpha _{n}^{1}\delta \alpha _{n}^{2}\delta \alpha _{n}^{3}\varepsilon
+\alpha _{n}^{1}\delta \left( 1-\alpha _{n}^{2}\right) \varepsilon +\alpha
_{n}^{1}\delta \alpha _{n}^{2}\varepsilon +\left( 1-\alpha _{n}^{1}\right)
\varepsilon +\alpha _{n}^{1}\varepsilon  \label{eqn49}
\end{eqnarray}%
Since $\delta \in \left( 0,1\right) $, $\alpha _{n}^{i}\in \left[ 0,1\right] 
$ for each $i\in \left\{ 1,2,3\right\} $ and for all $n\in 
\mathbb{N}
$,%
\begin{equation}
1-\alpha _{n}^{3}\left( 1-\delta \right) <1\text{,}  \label{eqn50}
\end{equation}%
\begin{equation}
1-\alpha _{n}^{2}\left( 1-\delta \right) <1\text{,}  \label{eqn51}
\end{equation}%
\begin{equation}
\alpha _{n}^{2}\alpha _{n}^{3}\delta ^{2}<1\text{,}  \label{eqn52}
\end{equation}%
\begin{equation}
\left( 1-\alpha _{n}^{2}\right) \delta <1\text{,}  \label{eqn53}
\end{equation}%
\begin{equation}
\alpha _{n}^{2}\delta <1\text{,}  \label{eqn54}
\end{equation}%
and by assumption (i) we have%
\begin{equation}
1-\alpha _{n}^{1}\leq \alpha _{n}^{1}\text{.}  \label{eqn55}
\end{equation}%
Thus, an application of inequalities (2.38), (2.39), (2.40), (2.41), (2.42)
and (2.43) to (2.37) yields%
\begin{equation}
\left\Vert p_{n+1}-\widetilde{p}_{n+1}\right\Vert \leq \left[ 1-\alpha
_{n}^{1}\left( 1-\delta \right) \right] \left\Vert p_{n}-\widetilde{p}%
_{n}\right\Vert +\alpha _{n}^{1}\left( 1-\delta \right) \frac{5\varepsilon }{%
1-\delta }\text{.}  \label{eqn56}
\end{equation}%
Let us denote%
\begin{equation}
a_{n}:=\left\Vert p_{n}-\widetilde{p}_{n}\right\Vert \text{, }\mu
_{n}:=\alpha _{n}^{1}\left( 1-\delta \right) \in \left( 0,1\right) \text{.}
\label{eqn57}
\end{equation}
It follows from Lemma 2 that%
\begin{equation}
0\leq \underset{n\rightarrow \infty }{\lim \sup }\left\Vert p_{n}-\widetilde{%
p}_{n}\right\Vert \leq \underset{n\rightarrow \infty }{\lim \sup }\frac{%
5\varepsilon }{1-\delta }\text{.}  \label{eqn58}
\end{equation}%
From Theorem 1 we know that $\lim_{n\rightarrow \infty }p_{n}=x_{\ast }$.
Thus, using this fact together with the assumption $\lim_{n\rightarrow
\infty }\widetilde{p}_{n}=u_{\ast }$ we obtain%
\begin{equation}
\left\Vert x_{\ast }-u_{\ast }\right\Vert \leq \frac{5\varepsilon }{1-\delta 
}\text{.}  \label{eqn59}
\end{equation}
\end{proof}

\end{document}